\documentclass{amsart}            

\usepackage{amssymb,amsmath,amsthm,latexsym,stmaryrd}

\usepackage[pagebackref,colorlinks=true]{hyperref}

\newcommand{\ie}{i.e. }
\newcommand{\Id}{\mathrm{Id}}
\newcommand{\F}{\mathcal{F}}
\newcommand{\X}{\mathfrak X}

\newcommand{\M}{(M,\allowbreak{}\f,\allowbreak{}\xi,\allowbreak{}\eta,\allowbreak{}g)}

\newcommand{\LL}{\mathcal{L}}
\newcommand{\R}{\mathbb{R}}
\newcommand{\Z}{\mathbb{Z}}

\newcommand{\n}{\nabla}

\newcommand{\f}{\varphi}
\newcommand{\D}{{\rm d}}

\newcommand{\al}{\alpha}

\newcommand{\lm}{\lambda}
\newcommand{\ta}{\theta}
\newcommand{\om}{\omega}

\newcommand{\norm}[1]{\Vert#1\Vert ^2}

\newcommand{\sgn}{\mathrm{sgn}}

\numberwithin{equation}{section}
\newtheorem{thm}{Theorem}[section]

\begin{document}
\title[Hyperspheres in real pseudo-Riemannian 4-spaces \ldots ]
{Space-like and time-like hyperspheres \\
in real pseudo-Riemannian 4-spaces \\
with almost contact B-metric structures}

\author{Hristo Manev$^{1,2}$}
\address[1]{Medical University of Plovdiv, Faculty of Pharmacy,
Department of Pharmaceutical Sciences,   15-A Vasil Aprilov
Blvd.,   Plovdiv 4002,   Bulgaria;}  \address[2]{Paisii
Hilendarski University of Plovdiv,   Faculty of Mathematics and
Informatics,   Department of Algebra and Geometry,   236
Bulgaria Blvd.,   Plovdiv 4027,   Bulgaria}
\email{hmanev@uni-plovdiv.bg}

\begin{abstract}
There are considered 4-dimensional pseudo-Riemannian spaces with inner products of signature (3,1) and (2,2). The objects of investigation are space-like and time-like hyperspheres in the respective cases. These hypersurfaces are equipped with almost contact B-metric structures. The constructed manifolds are characterized geometrically.
\end{abstract}


\maketitle

\section*{Introduction}

The geometry of 4-dimensional Riemannian spaces is well developed. When the metric is generalized to  pseudo-Riemannian there are two significant cases: the Lorentz-Minkowski space $\R^{3,1}$ and the neutral pseudo-Euclidean 4-space $\R^{2,2}$. These spaces are object of special interest because of their importance in physics. The space $\R^{3,1}$ has applications in the general relativity and the space $\R^{2,2}$ is connected to the string theory.

Hyperspheres in an even-dimensional space are known as a fundamental example of almost contact metric manifolds (cf. \cite{Blair}).
We are interested in almost contact B-metric structures, introduced in \cite{GaMiGr}.
In the present work we consider space-like and time-like hyperspheres in $\R^{3,1}$ and $\R^{2,2}$, known also as 3-dimensional de Sitter and anti-de Sitter space-times, respectively (cf. \cite{ChenVeken09}). After that we construct almost contact B-metric manifolds on these hypersurfaces. Then we study some their geometrical properties.

The paper\footnote{This paper is partially supported by a project of the Scientific Research Fund, Plovdiv University, Bulgaria} is organized as follows.
In Sect.~\ref{sect-prel} we recall some preliminary facts about the considered manifolds.
In Sect.~\ref{R3-1} we are interested in space-like spheres in $\R^{3,1}$.
Sect.~\ref{R2-2} is devoted to time-like spheres in $\R^{2,2}$.

\section{Preliminaries}\label{sect-prel}

Let us denote an \emph{almost contact B-metric manifold} by $(M,\f,\xi,\eta,g)$, \ie $M$
is a $(2n+1)$-dimensional differentiable manifold with an almost
contact structure $(\f,\xi,\eta)$ consisting of an endomorphism
$\f$ of the tangent bundle, a Reeb vector field $\xi$, its dual contact  1-form
$\eta$ as well as $M$ is equipped with a pseu\-do-Rie\-mannian
metric $g$  of signature $(n+1,n)$, such that the following
algebraic relations are satisfied \cite{GaMiGr}:
\begin{equation*}\label{strM}
\begin{array}{c}
\f\xi = 0,\qquad \f^2 = -\Id + \eta \otimes \xi,\qquad
\eta\circ\f=0,\qquad \eta(\xi)=1,\\[4pt]
g(\f x, \f y) = - g(x,y) + \eta(x)\eta(y),
\end{array}
\end{equation*}
where $\Id$ is the identity. In the latter equality and further, $x$, $y$, $z$, $w$ will stand for arbitrary elements of $\X(M)$, the Lie algebra of tangent vector fields, or vectors in the tangent space $T_pM$ of $M$ at an arbitrary
point $p$ in $M$.

A classification of almost contact B-metric manifolds, consisting of eleven basic classes $\F_1$, $\F_2$, $\dots$, $\F_{11}$, is given in
\cite{GaMiGr}. This classification is made with respect
to the tensor $F$ of type (0,3) defined by
\begin{equation*}\label{F=nfi}
F(x,y,z)=g\bigl( \left( \nabla_x \f \right)y,z\bigr),
\end{equation*}
where $\n$ is the Levi-Civita connection of $g$.
The following properties are valid in general:
\begin{equation}\label{F-prop}
\begin{array}{l}
F(x,y,z)=F(x,z,y)=F(x,\f y,\f z)+\eta(y)F(x,\xi,z)
+\eta(z)F(x,y,\xi),\\
F(x,\f y, \xi)=(\n_x\eta)y=g(\n_x\xi,y).
\end{array}
\end{equation}

The intersection of the basic classes is the special class $\F_0$,
determined by the condition $F(x,y,z)=0$, and it is known as the
class of the \emph{cosymplectic B-metric manifolds}.

Let $\left\{\xi;e_i\right\}$ $(i=1,2,\dots,2n)$ be a basis of
$T_pM$ and let $\left(g^{ij}\right)$ be the inverse matrix of
$\left(g_{ij}\right)$. Then with $F$ are associated the 1-forms
$\theta$, $\theta^*$, $\omega$, called \emph{Lee forms}, defined
by:
\begin{equation*}\label{t}
\theta(z)=g^{ij}F(e_i,e_j,z),\quad \theta^*(z)=g^{ij}F(e_i,\f
e_j,z), \quad \omega(z)=F(\xi,\xi,z).
\end{equation*}

Now let us consider the case of the lowest dimension of the considered manifolds, \ie $\dim{M}=3$.

We introduce an almost contact B-metric structure $(\f,\xi,\eta,g)$ on $M$ defined by
\begin{equation}\label{str}
\begin{array}{c}
\f e_1=0,\qquad \f e_2=e_{3},\qquad \f e_{3}=- e_2,\qquad \xi=
e_1,\\
\eta(e_1)=1,\qquad \eta(e_2)=\eta(e_{3})=0,
\end{array}
\end{equation}
\begin{equation}\label{gij}
g(e_1,e_1)=g(e_2,e_2)=-g(e_3,e_3)=1,\qquad g(e_i,e_j)=0,\;\; i\neq j \in \{1,2,3\}.
\end{equation}

The components of $F$, $\ta$, $\ta^*$, $\om$ with respect to the \emph{$\f$-basis}
$\left\{e_1,e_2,e_3\right\}$ are denoted by $F_{ijk}=F(e_i,e_j,e_k)$, $\ta_k=\ta(e_k)$, $\ta^*_k=\ta^*(e_k)$, $\om_k=\om(e_k)$.
According to \cite{HM1}, we have:
\[
\begin{array}{lll}
\ta_1=F_{221}-F_{331},\qquad &\ta_2=F_{222}-F_{332},\qquad &\ta_3=F_{223}-F_{322},\\[0pt]
\ta^*_1=F_{231}+F_{321},\qquad &\ta^*_2=F_{223}+F_{322},\qquad &\ta^*_3=F_{222}+F_{332},\\[0pt]
\om_1=0,\qquad &\om_2=F_{112},\qquad &\om_3=F_{113}.
\end{array}
\]

If $F^s$ $(s=1,2,\dots,11)$ are the components of $F$ in the
corresponding basic classes $\F_s$ then: \cite{HM1}
\begin{equation}\label{Fi3}
\begin{array}{l}
F^{1}(x,y,z)=\left(x^2\ta_2-x^3\ta_3\right)\left(y^2z^2+y^3z^3\right),\\[0pt]
\qquad\ta_2=F_{222}=F_{233},\qquad \ta_3=-F_{322}=-F_{333}; \\[0pt]
F^{2}(x,y,z)=F^{3}(x,y,z)=0;
\\[0pt]
F^{4}(x,y,z)=\frac{1}{2}\ta_1\Bigl\{x^2\left(y^1z^2+y^2z^1\right)
-x^3\left(y^1z^3+y^3z^1\right)\bigr\},\\[0pt]
\qquad \frac{1}{2}\ta_1=F_{212}=F_{221}=-F_{313}=-F_{331};\\[0pt]
F^{5}(x,y,z)=\frac{1}{2}\ta^*_1\bigl\{x^2\left(y^1z^3+y^3z^1\right)
+x^3\left(y^1z^2+y^2z^1\right)\bigr\},\\[0pt]
\qquad \frac{1}{2}\ta^*_1=F_{213}=F_{231}=F_{312}=F_{321};\\[0pt]
F^{6}(x,y,z)=F^{7}(x,y,z)=0;\\[0pt]
F^{8}(x,y,z)=\lm\bigl\{x^2\left(y^1z^2+y^2z^1\right)
+x^3\left(y^1z^3+y^3z^1\right)\bigr\},\\[0pt]
\qquad \lm=F_{212}=F_{221}=F_{313}=F_{331};\\[0pt]
F^{9}(x,y,z)=\mu\bigl\{x^2\left(y^1z^3+y^3z^1\right)
-x^3\left(y^1z^2+y^2z^1\right)\bigr\},\\[0pt]
\qquad \mu=F_{213}=F_{231}=-F_{312}=-F_{321};\\[0pt]
F^{10}(x,y,z)=\nu x^1\left(y^2z^2+y^3z^3\right),\qquad
\nu=F_{122}=F_{133};\\[0pt]
F^{11}(x,y,z)=x^1\bigl\{\left(y^2z^1+y^1z^2\right)\om_{2}
+\left(y^3z^1+y^1z^3\right)\om_{3}\bigr\},\\[0pt]
\qquad \om_2=F_{121}=F_{112},\qquad \om_3=F_{131}=F_{113},
\end{array}
\end{equation}
where $x=x^ie_i$, $y=y^je_j$, $z=z^ke_k$.
Obviously, the class of 3-dimensional almost contact B-metric
manifolds is
\[
\F_1 \oplus \F_4 \oplus \F_5 \oplus \F_8 \oplus \F_9 \oplus
\F_{10} \oplus \F_{11}.
\]

In \cite{Man31} are considered three natural connections on an arbitrary $\M$, i.e. linear connections which preserve $\f$, $\xi$, $\eta$, $g$. They are called a $\f$B-connection, a $\f$-canonical connection and a $\f$KT-connection. The $\f$B-connection is defined by
\begin{equation}\label{fiB}
{D}_xy=\n_xy+\frac{1}{2}\bigl\{\left(\n_x\f\right)\f
y+\left(\n_x\eta\right)y\cdot\xi\bigr\}-\eta(y)\n_x\xi.
\end{equation}
The $\f$-canonical connection is determined by an identity for its torsion with respect to the structure tensors and the $\f$KT-connection is characterized as the natural connection with totally antisymmetric torsion.

Since the considered manifold is 3-dimensional and the class $\F_3 \oplus \F_7$ is empty, then the $\f$KT-connection does not exist and the $\f$-canonical connection coincides with the $\f$B-connection.

In \cite{Man31} is defined the square norm of $\nabla \f$ as follows
\begin{equation}\label{snf}
    \norm{\nabla \f}=g^{ij}g^{ks}
    g\bigl(\left(\nabla_{e_i} \f\right)e_k,\left(\nabla_{e_j}
    \f\right)e_s\bigr).
\end{equation}
An almost contact B-metric manifold having a zero square
norm of $\n\f$ is called an
\emph{isotropic-cosymplectic B-metric manifold} (\cite{Man31}).
Obviously, the equality $\norm{\nabla \f}=0$ is valid if $\M$
is a $\F_0$-manifold, but the inverse implication is not always
true.

The Nijenhuis tensor $N$ of the almost contact structure is
defined as usual by $N = [\f, \f]+ \D{\eta}\otimes\xi$,
where $[\f,\f](x, y)=\left[\f x,\f
y\right]+\f^2\left[x,y\right]-\f\left[\f
x,y\right]-\f\left[x,\f y\right]$ for
$\left[x,y\right]=\n_xy-\n_yx$ and $\D\eta$ is the exterior
derivative of $\eta$.
According to \cite{ManIv36}, the
associated Nijenhuis tensor $\widehat{N}$  has the following form
$\widehat{N}=\{\f,\f\}+\left(\LL_{\xi}g\right)\otimes\xi$, where
$\{\f,\f\}(x, y)=\{\f x,\f y\}+\f^2\{x,y\}-\f\{\f
x,y\}-\f\{x,\f  y\}$ for $\left\{x,y\right\}=\n_xy+\n_yx$ and
$\LL_{\xi}g$ is the Lie derivative of $g$ with respect to $\xi$.

The corresponding tensors of type (0,3) on $\M$ are determined by
$N(x,y,z)=g\left(N(x,y),z\right)$ and $\widehat{N}(x,y,z)=g(
\widehat{N}(x,y),z)$. According to \cite{ManIv36}, it is known
that the tensors $N(x,y,z)$ and $\widehat{N}(x,y,z)$ are expressed by $F$ as follows
\begin{equation}\label{N-hatN-F}
\begin{split}
N(x,y,z)&=F(\f x,y,z)-F(x,y,\f z)+\eta(z)F(x,\f y,\xi)\\
& -F(\f y,x,z)+F(y,x,\f z)-\eta(z)F(y,\f x,\xi),\\
\widehat{N}(x,y,z)&=F(\f x,y,z)-F(x,y,\f z)+\eta(z)F(x,\f y,\xi)\\
& +F(\f y,x,z)-F(y,x,\f z)+\eta(z)F(y,\f x,\xi).
\end{split}
\end{equation}

Let $R=\left[\n,\n\right]-\n_{[\ ,\ ]}$ be the
curvature (1,3)-tensor of $\nabla$ and the corresponding curvature
$(0,4)$-tensor be denoted by the same letter: $R(x,y,z,w)$
$=g(R(x,y)z,w)$. The following properties are valid in general:
\begin{equation}\label{R}
\begin{array}{l}
R(x,y,z,w)=-R(y,x,z,w)=-R(x,y,w,z), \\
R(x,y,z,w)+R(y,z,x,w)+R(z,x,y,w)=0.
\end{array}
\end{equation}

The Ricci
tensor $\rho$ and the scalar curvature $\tau$ for $R$ and $g$ as well as
their associated quantities are defined as follows
\begin{equation}\label{rhotau}
\begin{array}{c}
    \rho(y,z)=g^{ij}R(e_i,y,z,e_j),\qquad \rho^*(y,z)=g^{ij}R(e_i,y,z,\f e_j),\\
    \tau=g^{ij}\rho(e_i,e_j),\quad
    \tau^*=g^{ij}\rho^*(e_i,e_j),\quad
    \tau^{**}=g^{ij}\rho^*(e_i,\f e_j).
\end{array}
\end{equation}

Each non-degenerate 2-plane $\al$ in
$T_pM$ with respect to $g$ and $R$ has the following sectional curvature
\begin{equation}\label{sect}
k(\al;p)=\frac{R(x,y,y,x)}{g(x,x)g(y,y)},
\end{equation}
where $\{x,y\}$ is an orthogonal basis of $\al$.

A 2-plane $\al$ is said to be a \emph{$\f$-holomorphic section}
(respectively, a \emph{$\xi$-section}) if $\al= \f\al$
(respectively, $\xi \in \al$).

\section{Space-like hyperspheres in $\R^{3,1}$}\label{R3-1}

In this section we consider a hypersurface of the Lorentz-Minkowski space $\R^{3,1}$. Let $\langle\cdot,\cdot\rangle$ be the Lorentzian inner product, i.e.
\begin{equation*}\label{inprod3-1}
  \langle x,y\rangle=x^1y^1+x^2y^2+x^3y^3-x^4y^4,
\end{equation*}
where $x(x^1,x^2,x^3,x^4)$, $y(y^1,y^2,y^3,y^4)$ are arbitrary vectors in $\R^{3,1}$.
Let us consider a space-like hypersphere $S^3_1$ at the origin with a real radius $r$ identifying the point $p$ in $\R^{3,1}$ with its position vector $z$, i.e.
\[
\langle z,z\rangle=r^2.
\]
It is parameterized by
\begin{equation*}\label{S1}
{z}(r\cos u^1\cos u^2, r\cos u^1\sin u^2, r\sin u^1\cosh u^3,  r\sin u^1\sinh u^3),
\end{equation*}
where $u^1,u^2,u^3$ are real parameters such as $u^1\neq\frac{k\pi}{2} (k \in \Z)$, $u^2 \in [0;2\pi]$.
Then for the local basic vectors $\partial_i=\frac{\partial z}{\partial{u^i}}$ we have the following
\begin{equation*}\label{S1-xx}
\begin{array}{l}
  \langle\partial_1,\partial_1\rangle=r^2, \quad
  \langle\partial_2,\partial_2\rangle=r^2 \cos^2u^1, \quad
  \langle\partial_3,\partial_3\rangle=-r^2 \sin^2u^1, \\
  \langle\partial_i,\partial_j\rangle=0, \; i\neq j.
\end{array}
\end{equation*}

By substituting $e_i=\frac{1}{\sqrt{\left|\langle\partial_i,\partial_i\rangle\right|}}\partial_i$  we obtain a basis $\{e_i\}$, $i \in \{1,2,3\}$ as follows
\begin{equation}\label{S1-eidi}
\begin{array}{l}
{e_1}=\frac{1}{r}\partial_1, \qquad
{e_2}=\frac{\varepsilon_1}{r\cos u^1}\partial_2, \qquad
{e_3}=\frac{\varepsilon_2}{r\sin u^1}\partial_3,
\end{array}
\end{equation}
where $\varepsilon_1=\sgn (\cos u^1)$, $\varepsilon_2=\sgn (\sin u^1)$.
We equip it with an almost contact structure determined as in \eqref{str}.
The metric on the hypersurface, denoted by $g$, is the
restriction of $\langle\cdot,\cdot\rangle$ on the sphere. Then $\{e_1,e_2,e_3\}$ is an orthonormal $\f$-basis on the tangent space $T_pS^3_1$ at $p \in S^3_1$, i.e. for $g_{ij}=g(e_i,e_j)$, $i,j \in \{1,2,3\}$, we have \eqref{gij}.
Thus, we get that $(S^3_1,\f,\xi,\eta,g)$ is a 3-dimensional almost contact B-metric manifold.

By virtue of \eqref{S1-eidi} we obtain the commutators of the basic vectors $e_i$
\begin{equation}\label{S1-com}
\begin{array}{l}
  [e_1,e_2]=\frac{1}{r}\tan u^1 {e_2}, \qquad [e_1,e_3]=-\frac{1}{r}\cot u^1 {e_3}, \qquad [e_2,e_3]=0.
\end{array}
\end{equation}
Using the well-known Koszul identity for $\n$ of $g$ we get
\begin{equation}\label{S1-nij}
\begin{array}{ll}
\n_{e_2}e_1=-\frac{1}{r}\tan u^1 {e_2}, \qquad &\n_{e_2}e_2=\frac{1}{r}\tan u^1 {e_1},\\
\n_{e_3}e_1=\frac{1}{r}\cot u^1 {e_3}, \qquad &\n_{e_3}e_3=\frac{1}{r}\cot u^1 {e_1}
\end{array}
\end{equation}
and the other components are zero.

Let us compute the components of the natural connection denoted by $D$ in \eqref{fiB}.
Then, using \eqref{str}, \eqref{gij}, \eqref{fiB}, \eqref{S1-nij}, we establish that
\begin{equation}\label{S1-Dij}
  D_{e_i}e_j=0, \qquad i,j \in \{1,2,3\}.
\end{equation}

According to \eqref{str}, \eqref{gij} and \eqref{S1-nij}, we obtain the value of the square norm of $\n \f$ as follows
\begin{equation}\label{S1-nf}
\norm{\nabla \f}=-\frac{2}{r^2}(\tan^2 u^1 + \cot^2 u^1).
\end{equation}

Taking into account \eqref{str}, \eqref{gij} and \eqref{S1-nij}, we compute the components $F_{ijk}$ of $F$ with respect to the basis $\{e_1,e_2,e_3\}$. They are
\begin{equation}\label{S1-Fijk}
  F_{213}=F_{231}=-\frac{1}{r}\tan u^1, \qquad  F_{312}=F_{321}=\frac{1}{r}\cot u^1
\end{equation}
and the other components of $F$ are zero.

Using \eqref{N-hatN-F} and \eqref{S1-Fijk}, we find the basic components $N_{ijk}=N(e_i,e_j,e_k)$ and $\widehat{N}_{ijk}=\widehat{N}(e_i,e_j,e_k)$ of the Nijenhuis tensor and its associated tensor, respectively,
\begin{equation*}
\begin{array}{l}
N_{122}=-N_{212}=N_{133}=-N_{313}=-\frac{1}{r}(\cot u^1+\tan u^1), \\
\widehat{N}_{122}=\widehat{N}_{212}=\widehat{N}_{133}=\widehat{N}_{313}=\frac{1}{r}(\cot u^1+\tan u^1), \\
\widehat{N}_{221}=-\widehat{N}_{331}=-\frac{4}{r}\tan u^1,
\end{array}
\end{equation*}
as well as their square norms, according to \eqref{snf}, as follows
\begin{equation}\label{S1-NNhat}
\begin{array}{l}
\norm{N}=\frac{4}{r^2}(\cot^2 u^1+\tan^2 u^1+2), \\
\norm{\widehat{N}}=\frac{4}{r^2}(\cot^2 u^1+9\tan^2 u^1+2).
\end{array}
\end{equation}

Bearing in mind \eqref{Fi3} and \eqref{S1-Fijk}, we establish the equality
\begin{equation}\label{S1-F5+F9}
F(x,y,z)=(F^5+F^9)(x,y,z),
\end{equation}
where $F^5$ and $F^9$ are the components of $F$ in the basic classes $\F_5$ and $\F_9$, respectively. The nonzero components of $F^5$ and $F^9$ with respect to $\{e_1,e_2,e_3\}$ are the following
\begin{equation}\label{S1-Fijk59}
\begin{array}{l}
F^5_{213}=F^5_{231}=F^5_{312}=F^5_{321}=\frac{1}{2}\ta^*_1=\frac{1}{2r}(\cot u^1-\tan u^1),\\
F^9_{213}=F^9_{231}=-F^9_{312}=-F^9_{321}=\mu=-\frac{1}{2r}(\cot u^1+\tan u^1).
\end{array}
\end{equation}
Let us remark that the above components of $F^5$ and $F^9$ are nonzero for all values of $u^1$ in its domain.
By virtue of \eqref{S1-F5+F9}, \eqref{S1-Fijk59} and \eqref{F-prop},  we get that
\begin{equation}\label{S1-d-eta}
d\eta=0, \qquad \n_\xi \xi=0.
\end{equation}

Using \eqref{gij}, \eqref{S1-com} and \eqref{S1-nij}, we compute the components $R_{ijk\ell}=R(e_i,e_j,e_k,e_\ell)$ of the curvature tensor $R$ with respect to $\{e_1,e_2,e_3\}$. The nonzero components are given by the following ones and the symmetries of $R$ in \eqref{R}
\begin{equation}\label{S1-Rijkl}
  R_{1221}=-R_{1331}=-R_{2332}=\frac{1}{r^2}.
\end{equation}

By virtue of \eqref{gij}, \eqref{rhotau} and \eqref{S1-Rijkl}, the basic components $\rho_{jk}=\rho(e_j,e_k)$ and $\rho^*_{jk}=\rho^*(e_j,e_k)$ of the
Ricci tensor $\rho$ and its associated tensor $\rho^*$, respectively, as well as the values
of the scalar curvature $\tau$ and its associated curvatures $\tau^*$, $\tau^{**}$ are
the following
\begin{equation*}\label{S1-rhotau}
\begin{array}{l}
\rho_{11}=\rho_{22}=-\rho_{33}=\frac{2}{r^2}, \quad \rho^{*}_{23}=\rho^{*}_{32}=\frac{1}{r^2},\\
\tau=\frac{6}{r^2}, \qquad \tau^{*}=0, \qquad \tau^{**}=\frac{2}{r^2}.
\end{array}
\end{equation*}
Moreover, using \eqref{gij}, \eqref{sect} and \eqref{S1-Rijkl}, we obtain the basic sectional curvatures $k_{ij}=k(e_i,e_j)$ determined by the basis $\{e_i,e_j\}$ of the corresponding 2-plane as follows
\begin{equation}\label{S1-kij}
k_{12}=k_{13}=k_{23}=\frac{1}{r^2}.
\end{equation}

Let us remark that \eqref{gij}, \eqref{S1-Rijkl} and \eqref{S1-kij} imply the following form of the curvature tensor
\begin{equation}\label{Rpi1}
R(x,y,z,w)=\frac{1}{r^2}\{g(y,z)g(x,w)-g(x,z)g(y,w)\}.
\end{equation}

Bearing in mind the above results, we establish the truthfulness of the following
\begin{thm}
Let $(S^3_1,\f,\xi,\eta,g)$ be the space-like sphere in the Lorentz-Minkowski space $\R^{3,1}$ equipped with an almost contact B-metric structure.
Then
\begin{enumerate}
\item
the manifold is in the class $\F_5\oplus\F_9$ but it belongs neither to $\F_5$ nor $\F_9$ and it is not an isotropic-cosymplectic B-metric manifold;
\item
the $\f$B-connection which coincides with the $\f$-canonical connection vanishes in the basis $\{e_1,e_2,e_3\}$;
\item
the square norm of $\nabla \f$ is negative;
\item
the square norms of the Nijenhuis tensor and its associated are positive;
\item
the contact form $\eta$ is closed and the integral curves of $\xi$ are geodesic;
\item
the manifold is a space-form with positive constant sectional curvature.
\end{enumerate}
\end{thm}

\begin{proof}
The proposition (1) follows from \eqref{S1-nf}, \eqref{S1-F5+F9} and \eqref{S1-Fijk59}. The truthfulness of the propositions (2), (3), (4), (5), (6) follows from \eqref{S1-Dij}, \eqref{S1-nf}, \eqref{S1-NNhat}, \eqref{S1-d-eta}, \eqref{Rpi1}, respectively.
\end{proof}

\section{Time-like hyperspheres in $\R^{2,2}$}\label{R2-2}
In \cite{GaMiGr}, it is considered a unit time-like hypersphere $S$ in $(\R^{2n+2},J,G)$, where $\R^{2n+2}$ is a complex Riemannian manifold with a canonical complex structure $J$ and a Norden metric $G$. There is introduced an almost contact B-metric structure on $S$ in appropriate way by means of $J$ and $G$. The constructed hypersphere with the considered structure belongs to the class $\F_4 \oplus \F_5$.

In this section we use a different approach for equipping a time-like hypersphere in $\R^{2n+2}$ for $n=1$ with an almost contact B-metric structure.

Let us consider the neutral pseudo-Euclidean 4-space $\R^{2,2}$.
Let $\langle\cdot,\cdot\rangle$ be the inner product defined by
\begin{equation*}\label{inprod2-2}
  \langle x,y\rangle=x^1y^1+x^2y^2-x^3y^3-x^4y^4
\end{equation*}
for arbitrary vectors $x(x^1,x^2,x^3,x^4)$, $y(y^1,y^2,y^3,y^4)$ in $\R^{2,2}$.
Let us consider a time-like hypersphere $H^3_1$ at the origin with a real radius $r$ identifying the point $p$ in $\R^{2,2}$ with its position vector $z$, i.e.
\[
\langle z,z\rangle=-r^2.
\]
It is parameterized by
\begin{equation*}\label{S2}
{z}(r\sinh u^1\cos u^2, r\sinh u^1\sin u^2, r\cosh u^1\cos u^3,  r\cosh u^1\sin u^3),
\end{equation*}
where $u^1,u^2,u^3 \in \R$ such as $u^1\neq0$.
Then, for the local basic vectors $\partial_i$, we have the following
\begin{equation*}\label{S2-xx}
\begin{array}{l}
  \langle\partial_1,\partial_1\rangle=r^2, \quad
  \langle\partial_2,\partial_2\rangle=r^2 \sinh^2 u^1, \quad
  \langle\partial_3,\partial_3\rangle=-r^2 \cosh^2 u^1, \\
  \langle\partial_i,\partial_j\rangle=0, \; i\neq j.
\end{array}
\end{equation*}

Similarly as in the previous section, we substitute $e_i=\frac{1}{\sqrt{\left|\langle\partial_i,\partial_i\rangle\right|}}\partial_i$  and we obtain an orthonormal basis $\{e_i\}$, $i \in \{1,2,3\}$, as follows
\begin{equation*}\label{S2-eidi}
\begin{array}{l}
{e_1}=\frac{1}{r}\partial_1, \qquad
{e_2}=\frac{\varepsilon}{r\sinh u^1}\partial_2, \qquad
{e_3}=\frac{1}{r\cosh u^1}\partial_3,
\end{array}
\end{equation*}
where $\varepsilon=\sgn (u^1)$.
As for $S^3_1$, we introduce an almost contact B-metric structure on $H^3_1$ determined by \eqref{str} and \eqref{gij}.
Hence, we get that $(H^3_1,\f,\xi,\eta,g)$ is a 3-dimensional almost contact B-metric manifold.

By similar way as for $S^3_1$ we obtain successively the following results:
\begin{equation*}\label{S2-com}
\begin{array}{c}
[e_1,e_2]=-\frac{1}{r}\coth u^1 {e_2}, \qquad [e_1,e_3]=-\frac{1}{r}\tanh u^1 {e_3}, \qquad [e_2,e_3]=0,
\end{array}
\end{equation*}
\begin{equation*}\label{S2-nij}
\begin{array}{l}
\n_{e_2}e_1=\frac{1}{r}\coth u^1 {e_2}, \qquad \n_{e_2}e_2=-\frac{1}{r}\coth u^1 {e_1},\\
\n_{e_3}e_1=\frac{1}{r}\tanh u^1 {e_3}, \qquad \n_{e_3}e_3=\frac{1}{r}\tanh u^1 {e_1},
\end{array}
\end{equation*}
\begin{equation}\label{S2-Dij}
  D_{e_i}e_j=0, \qquad i,j \in \{1,2,3\},
\end{equation}
\begin{equation}\label{S2-nf}
\begin{array}{c}
\norm{\nabla \f}=-\frac{2}{r^2}(\tanh^2 u^1 + \coth^2 u^1),
\end{array}
\end{equation}
\begin{equation*}\label{S2-Fijk}
\begin{array}{c}
F_{213}=F_{231}=\frac{1}{r}\coth u^1, \qquad  F_{312}=F_{321}=\frac{1}{r}\tanh u^1,
\end{array}
\end{equation*}
\begin{equation*}
\begin{array}{l}
N_{122}=-N_{212}=N_{133}=-N_{313}=\frac{2}{r\sinh 2u^1}, \\
\widehat{N}_{122}=\widehat{N}_{212}=\widehat{N}_{133}=\widehat{N}_{313}=-\frac{2}{r\sinh 2u^1}, \\
\widehat{N}_{221}=-\widehat{N}_{331}=\frac{2}{r}(\coth u^1 + \tanh u^1),
\end{array}
\end{equation*}
\begin{equation}\label{S2-NNhat}
\begin{array}{l}
\norm{N}=\frac{4}{r^2}(\coth^2 u^1+\tanh^2 u^1+2), \\
\norm{\widehat{N}}=\frac{4}{r^2}(3\coth^2 u^1+3\tanh^2 u^1+2),
\end{array}
\end{equation}
\begin{equation}\label{S2-F5+F9}
F(x,y,z)=(F^5+F^9)(x,y,z),
\end{equation}
\begin{equation}\label{S2-Fijk59}
\begin{array}{l}
F^5_{213}=F^5_{231}=F^5_{312}=F^5_{321}=\frac{1}{2}\ta^*_1=\frac{1}{2r}(\coth u^1+\tanh u^1),\\
F^9_{213}=F^9_{231}=-F^9_{312}=-F^9_{321}=\mu=\frac{1}{2r}(\coth u^1-\tanh u^1),
\end{array}
\end{equation}
\begin{equation}\label{S2-d-eta}
d\eta=0, \qquad \n_\xi \xi=0,
\end{equation}
\begin{equation}\label{S2-kij}
\begin{array}{c}
R_{1221}=-R_{1331}=-R_{2332}=k_{12}=k_{13}=k_{23}=-\frac{1}{r^2},
\end{array}
\end{equation}
\begin{equation*}\label{S2-rhotau}
\begin{array}{l}
\rho_{11}=\rho_{22}=-\rho_{33}=-\frac{2}{r^2}, \qquad \rho^{*}_{23}=\rho^{*}_{32}=-\frac{1}{r^2},\\
\tau=-\frac{6}{r^2}, \qquad \tau^{*}=0, \qquad \tau^{**}=-\frac{2}{r^2}.
\end{array}
\end{equation*}

Similarly to the case of $S^3_1$, the obtained results could be interpreted in the following
\begin{thm}
Let $(H^3_1,\f,\xi,\eta,g)$ be the time-like sphere in the space $\R^{2,2}$ equipped with an almost contact B-metric structure.
Then
\begin{enumerate}
\item
the manifold is in the class $\F_5\oplus\F_9$ but it belongs neither to $\F_5$ nor $\F_9$ and it is not an isotropic-cosymplectic B-metric manifold;
\item
the $\f$B-connection which coincides with the $\f$-canonical connection vanishes in the basis $\{e_1,e_2,e_3\}$;
\item
the square norm of $\nabla \f$ is negative;
\item
the square norms of the Nijenhuis tensor and its associated are positive;
\item
the contact form $\eta$ is closed and the integral curves of $\xi$ are geodesic;
\item
the manifold is a space-form with negative constant sectional curvature.
\end{enumerate}
\end{thm}

\begin{proof}
The proposition (1) follows from \eqref{S2-nf}, \eqref{S2-F5+F9} and \eqref{S2-Fijk59}. The truthfulness of the propositions (2), (3), (4), (5), (6) follows from \eqref{S2-Dij}, \eqref{S2-nf}, \eqref{S2-NNhat}, \eqref{S2-d-eta}, \eqref{S2-kij}, respectively.
\end{proof}


\begin{thebibliography}{33}


\bibitem{Blair}
Blair, D.E.: \emph{Riemannian Geometry of Contact and Symplectic
Manifolds}. Pro\-gress in Mathematics \textbf{203}. Birkh\"auser,
Boston (2002)


\bibitem{ChenVeken09}
Chen, B.Y., Van der Veken, J.: \emph{Complete classification of
parallel surfaces in 4-dimensional
Lorentz space forms}. T\^{o}hoku Math. J. \textbf{61}, 1--40  (2009)


\bibitem{GaMiGr}
Ganchev, G., Mihova, V., Gribachev, K.: \emph{Almost contact
manifolds with B-metric}. Math. Balkanica (N.S.)  \textbf{7}
(3-4), 261--276 (1993)


\bibitem{HM1}
Manev, H.: \emph{On the structure tensors of almost contact
B-metric manifolds}. {arXiv:\allowbreak{}1405.3088}



\bibitem{Man31}
Manev, M.: \emph{Natural connection with totally skew-symmetric
tor\-sion on almost con\-tact manifolds with B-metric}. Int. J.
Ge\-om. Methods Mod. Phys. \textbf{9} (5), 1250044 (20 pa\-ges)
(2012)



\bibitem{ManIv36}
Manev, M., Ivanova, M.: \emph{A classification of the torsion
tensors on almost contact manifolds with B-metric}. Cent. Eur. J.
Math. \textbf{12} (10), 1416--1432 (2014)


\end{thebibliography}
\end{document}